\documentclass{article}

\usepackage[a4paper,top=2.54cm,bottom=2.54cm,left=3.17cm,right=3.17cm,%
            includehead,includefoot]{geometry}

\usepackage{amsmath,amssymb,amsfonts,amsthm}
\usepackage{graphicx}
\usepackage{float}
\usepackage{xcolor}
\usepackage[numbers,square,sort&compress]{natbib}
\usepackage{hyperref}
  \hypersetup{colorlinks,citecolor=blue,linkcolor=blue,breaklinks=true}
\usepackage{epstopdf}
\usepackage{amsthm,amsmath,amssymb}
\usepackage{rerunfilecheck}
\usepackage{bookmark}
\usepackage{mathrsfs}
\usepackage{rerunfilecheck}

\makeatletter

\newcommand{\Rmnum}[1]{\expandafter\@slowromancap\romannumeral #1@}
\makeatother

\usepackage{yhmath} 

\allowdisplaybreaks
\numberwithin{equation}{section}

\newtheorem{theorem}{Theorem}[section] 
\newtheorem{lemma}[theorem]{Lemma} 
\newtheorem{corollary}[theorem]{Corollary}

\newtheorem{remark}[theorem]{Remark}  

\usepackage{pdfpages}
\begin{document}

\title{\uppercase{gradient estimates and parabolic frequency monotonicity for positive solutions of the heat equation under generalized Ricci flow }}

\author{
  Juanling Lu
    \thanks{School of Mathematical Sciences, East China Normal University, 500 Dongchuan Road, Shanghai 200241, P. R. of China, E-mail address:
    lujuanlingmath@163.com}
  \and
  Yu Zheng
    \thanks{School of Mathematical Sciences, East China Normal University, 500 Dongchuan Road, Shanghai 200241, P. R. of China, E-mail address: zhyu@math.ecnu.edu.cn.}
   \footnote{This work was supported by the National Natural Science Foundation of China (No.~12271163), the Science and Technology Commission of Shanghai Municipality (No.~22DZ2229014), Key Laboratory of MEA (Ministry of Education) and Shanghai Key Laboratory of PMMP(No.~18DZ2271000).}
}

\date{}

\maketitle

\begin{abstract}
In this paper, we establish Li-Yau-type and Hamilton-type estimates for positive solutions to the heat equation associated with the generalized Ricci flow, under a less stringent curvature condition. Compared with \cite{liu shiping} and \cite{Zhang-Hamilton estimate}, these estimates generalize the results in Ricci flow to this new flow under the weaker Ricci curvature bounded assumption. As an application, we derive the Harnack-type inequalities in spacetime and find the monotonicity of one parabolic frequency for positive solutions of the heat equation under bounded Ricci curvature.
\end{abstract}

\section{Introduction}
\subsection{Gradient estimates under the generalized Ricci flow}
In their seminal paper \cite{Peter Li and Yau}, P. Li and S.-T. Yau developed fundamental gradient estimates for positive solutions to the heat equation on Riemannian manifolds. In particular, they demonstrated that if $u: \mathrm{\bf{M}}^{n} \times[0,\infty) \rightarrow \mathbb{R}$ is a positive solution to the heat equation $\partial_{t}u=\Delta u$ on an n-dimensional complete Riemannian manifold $(\mathrm{\bf{M}}^{n}, g)$ with nonnegative Ricci curvature, then it satisfies the following estimate
$$\frac{\partial_{t}u}{u}-\frac{|\nabla u|^{2}}{u^{2}}+\frac{n}{2t} \geq \Delta\ln u +\frac{n}{2t} \geq 0$$
for all $(x,t) \in \mathrm{\bf{M}}\times(0,\infty)$. Remarkably, the Li-Yau estimate is also referred to as a differential Harnack inequality as integrating it yields a sharp version of the classical Harnack inequality originally formulated from Moser \cite{Moser 1964}.

In \cite{Hamilton 1993}, when $(\mathrm{\bf{M}}^{n},g)$ is a closed n-dimensional manifold with Ricci curvature bounded below
by $\mathrm{Ric} \geq -Kg$ for some constant $K \geq 0$, R. S. Hamilton showed that positive solutions $u=u(x,t)$ to the heat equation, which satisfy the condition $u(x,t) \leq A$, adhere to the following gradient estimate
$$\frac{|\nabla u|^{2}}{u^{2}} \leq \left(\frac{1}{t}+2K\right)\ln\left(\frac{A}{u}\right).$$

The aforementioned estimates are of significant importance, and numerous scholars have conducted research on this topic. When metrics evolve under the Ricci flow, the Li-Yau-type gradient estimate for positive solutions of the heat equation has been established, as documented in \cite{cao xiaodong,liu shiping} among others. In 2015, B\u{a}ile\c{s}teanu \cite{Ricci Harmonic Flow} examined the Li-Yau estimate for positive solutions of the heat equation under the Ricci-harmonic flow, given the appropriate conditions. Recently, Li, Li, and Xu \cite{Li Yi-G2 flow} proved that the Li-Yau-type gradient estimate also holds when the metric evolves via the Laplacian $G_\text{2}$ flow on a closed 7-dimensional manifold with a closed $G_2$-structure. Additionally, there is a wealth of research on Hamilton-type gradient estimates for the heat equation under various geometric flows, as evidenced by \cite{Zhang-Hamilton estimate, Li Yi Ricci and Ricci harmonic flow, Li Yi-G2 flow} and others.

In this paper, we study the gradient estimates of positive solutions to the heat equation under the generalized Ricci flow. There have been many scholars who have done extensive research on this flow in different aspects, for details, see \cite{Streets-book, entropy functional-Japan, Streets-Bochners formula, Li Xilun, Streets-scalar and entropy} and so on.

This flow is described by the following equations
\begin{equation}\label{definition of generalized Ricii flow}
\begin{cases}
\partial_{t}g=-2\mathrm{Ric}+\dfrac{1}{2}H^{2}, \\
\partial_{t}H=-dd^{\ast}_{g}H,
\end{cases}
\end{equation}
where $H$ is a closed three-form on the manifold $(\mathrm{\bf{M}}^{n},g(t))$, $H^{2}$ is positive semidefinite tensor defined by
$$H^{2}(X,Y) =\langle i_{X}H,i_{Y}H \rangle_{g(t)},$$ with $i_{X}$ denoting the interior product and the inner product being taken with respect to the the time-dependent metric $g(t)$, $d^{\ast}_{g}$ represents the adjoint of the exterior differential $d$ acting on differential forms with respect to the metric $g(t)$. This equation arises independently across various fields, including mathematical physics \cite{Oliynyk-math physics}, complex geometry \cite{Tian Gang-complex geometry,Tian Gang and Streets-complex geometry}, and generalized geometry \cite{Garcia-generalized geometry,Tian Gang-generalized geometry,Streets-generalized geometry}. For additional background, we refer the reader to \cite{Streets-book}. It is noteworthy that the condition $H\equiv0$ is preserved by the flow (see \cite{Streets-book}, Proposition 4.20). In this case, the metric evolves according to the Ricci flow. Consequently, the remainder of this paper primarily focuses on results pertaining to the generalized Ricci flow, with the corresponding results for the classical Ricci flow emerging as a special case.

Here, we first extend the Li-Yau gradient estimate for the heat equation, given by
\begin{equation}\label{definition of heat eq.}
\begin{aligned}
\partial_{t}u=\Delta_{g(t)}u,
\end{aligned}
\end{equation}
to the case of the generalized Ricci flow described by (\ref{definition of generalized Ricii flow}), where $\Delta_{g(t)}=\mathrm{tr}_{g(t)}\left(\nabla_{g(t)}^{2}\right)$ is the usual Laplacian induced by $g(t)$.

\begin{theorem}\label{gradient estimates 1-1}
Let $(\mathrm{\bf{M}}^{n},g(t),H(t)),t\in[0,T]$, be the solution of the generalized Ricci flow (\ref{definition of generalized Ricii flow}) on an $n$-dimensional closed manifold $\mathrm{\bf{M}}$ with $-\dfrac{K_1}{t}g \leq \mathrm{Ric} \leq \dfrac{K_2}{t}g$ and $H^2\leq\dfrac{K_{3}}{t}g$, where $K_{1}$, $K_{2}$, $K_{3}$ and $T<\infty$ are positive constants. Assume that $|\nabla H^{2}| \leq K_{4}$ for some constant $K_{4}>0$. Suppose that $u:\mathrm{\bf{M}}\times[0,T]\rightarrow\mathbb{R}$ is a smooth positive solution of the heat equation (\ref{definition of heat eq.}), we then derive the following
\begin{equation*}
\dfrac{|\nabla u|^{2}}{u^{2}}-\alpha\dfrac{\partial_{t}u}{u}
\leq \sqrt{\frac{n\alpha}{2a}}\left[\left(\sqrt{\frac{n\alpha}{2a}}+\sqrt{B_{2}}\right)\frac{1}{t}+\frac{\sqrt{B_{3}}}{\sqrt{t}}+\sqrt{B_{1}}\right]
\end{equation*}
for any $\alpha>1$ and $a>0, ~b>0$ with $a+2b=\dfrac{1}{\alpha}$, where
$B_1 = \dfrac{n\alpha^3}{512a(\alpha-1)^2} + \dfrac{3n\alpha K_{4}^{2}}{8} $,
$B_2 =
\dfrac{n\alpha^3}{2a(\alpha-1)^2}\left(K_{1}+\dfrac{(\alpha-1)K_{3}}{4\alpha}\right)^{2} + \dfrac{n\alpha K}{2b} + \dfrac{n\alpha K_3^2}{32b}$,
$B_{3}=\dfrac{n\alpha^3}{16a(\alpha-1)^2} \left( K_1 + \dfrac{(\alpha-1)K_3}{4\alpha} \right)$,
$K=\mathrm{max}\{K_{1}^{2},K_{2}^{2}\}$.
\end{theorem}

\begin{remark}
It should be noted that the condition on $\alpha >1$ comes from the key estimate (\ref{the key range of a}) in the proof of the Theorem \ref{gradient estimates 1-1}. For $\alpha \leq 1$, it is still an open problem.
\end{remark}

\begin{remark}
Since $|\nabla H^{2}|$ is bounded on a closed manifold, $K_{4}$ exists and is finite. This shows that the assumption condition $|\nabla H^{2}| \leq K_{4}$ is natural.
\end{remark}

\begin{remark}
When the generalized Ricci flow satisfies $H(\cdot,0) = 0$, which indicates that the metric $g(t)$ evolves according to the Ricci flow (see \cite[Proposition 4.20]{Streets-book}), we can obtain the following estimate from Theorem \ref{gradient estimates 1-1}:
\begin{equation}\label{estimate of Ricci Flow}
\begin{aligned}
\dfrac{|\nabla u|^{2}}{u^{2}}-\alpha\dfrac{\partial_{t}u}{u}
\leq \dfrac{n\alpha}{2at} + \frac{1}{t}\sqrt{\frac{n^{2}\alpha^{4}}{4a^{2}(\alpha-1)^{2}}+\frac{n^{2}\alpha^{2}K}{4ab}}
\end{aligned}
\end{equation}
for any $\alpha>1$ and $a,b>0$ with $a+2b=\dfrac{1}{\alpha}$, where $K=\mathrm{max}\{K_{1}^{2},K_{2}^{2}\}$. It is evident that our estimate (\ref{estimate of Ricci Flow}) represents an improvement over Theorem 2 in \cite{liu shiping}, as we have weakened the condition $-K_{1}g \leq \mathrm{Ric} \leq K_{2}g$ to $-\dfrac{K_1}{t}g \leq \mathrm{Ric} \leq \dfrac{K_2}{t}g$.
\end{remark}

According to Theorem \ref{gradient estimates 1-1}, we obtain the following Harnack inequality.

\begin{corollary}\label{Harnack}
Under the same hypotheses as Theorem \ref{gradient estimates 1-1}, then the following holds
\begin{equation*}
u(x,t_{1}) \leq u(y,t_{2})\left(\frac{t_{2}}{t_{1}}\right)^{\frac{\xi}{\alpha}}
exp\left\{\int_{0}^{1} \frac{\alpha|\gamma'(s)|^{2}}{4(t_{2}-t_{1})} ds
+\frac{2}{\alpha}\sqrt{\frac{n\alpha}{2a}B_{3}}\left(\sqrt{t_{2}}-\sqrt{t_{1}}\right)
+\frac{t_{2}-t_{1}}{\alpha}\sqrt{\frac{n\alpha}{2a}B_{1}}\right\}
\end{equation*}
for any $x,y \in \mathrm{\bf{M}}$ and $0<t_{1}<t_{2} \leq T$, where $\alpha>1$, $\xi=\dfrac{n\alpha}{2a}+\sqrt{\dfrac{n\alpha}{2a}B_{2}}~$, $\gamma$ is the minimal geodesic connecting $x$ and $y$, and the constants $a,B_{1},B_{2},B_{3}$ are as shown in Theorem \ref{gradient estimates 1-1}.
\end{corollary}

\begin{remark}
When the metric $g(t)$ evolves under the Ricci flow, we derive the following Harnack-type inequality based on Corollary \ref{Harnack}:
\begin{equation}\label{Harnack of R.F.}
\begin{aligned}
u(x,t_{1}) \leq u(y,t_{2})\left(\frac{t_{2}}{t_{1}}\right)^{\frac{\xi}{\alpha}}
exp\left\{\int_{0}^{1} \frac{\alpha|\gamma'(s)|^{2}}{4(t_{2}-t_{1})} ds \right\}
\end{aligned}
\end{equation}
for any $x,y \in \mathrm{\bf{M}}$ and $0<t_{1}<t_{2} \leq T$, where $\alpha>1$, $\xi=\dfrac{n\alpha}{2a}+\sqrt{\dfrac{n^{2}\alpha^{4}}{4a^{2}(\alpha-1)^{2}}+\dfrac{n^{2}\alpha^{2}K}{4ab}}~$.
Clearly, inequality (\ref{Harnack of R.F.}) represents an improvement over Corollary 2 in \cite{liu shiping}.
\end{remark}

For the Hamilton-type gradient estimate of the heat equation (\ref{definition of heat eq.}) under the generalized Ricci flow (\ref{definition of generalized Ricii flow}), we have
\begin{theorem}\label{gradient estimates 1-2}
Let $(\mathrm{\bf{M}}^{n},g(t),H(t)),t\in[0,T]$, be the solution of the generalized Ricci flow (\ref{definition of generalized Ricii flow}) on an $n$-dimensional closed manifold $\mathrm{\bf{M}}$. Suppose that $u:\mathrm{\bf{M}}\times[0,T]\rightarrow\mathbb{R}$ is a smooth positive solution of the heat equation (\ref{definition of heat eq.}), then we obtain
\begin{equation*}
|\nabla u|^{2} \leq \frac{u^{2}}{t}\ln\left(\frac{A}{u}\right)
\end{equation*}
for any $A=\mathrm{max}_{\mathrm{\bf{M}}}u(\cdot,0)$ and $(x,t) \in \mathrm{\bf{M}}\times[0,T]$.
\end{theorem}

\begin{remark}
In fact, Theorem \ref{gradient estimates 1-2} extends Theorem 3.3 presented in \cite{Zhang-Hamilton estimate}, as the Ricci flow is a specific instance of the generalized Ricci flow.
\end{remark}
\subsection{Parabolic frequency under the generalized Ricci flow}
The elliptic frequency
$$N_{e}(r)=\frac{r\int_{B(r,p)} |\nabla u(x)|^{2} dx}{\int_{\partial B(p,r)} u^{2}(x) d\sigma }$$
for a harmonic function $u$ on $\mathbb{R}^{n}$ was introduced by Almgren \cite{Almgren} in 1979. Here, $d\sigma$ is the induced $(n-1)$-dimensional Hausdorff measure on $\partial B(r, p)$, where $B(r, p)$ represents the ball in $\mathbb{R}^{n}$ and $p$ is a fixed point in $\mathbb{R}^{n}$.
Almgren observed that $N_{e}(r)$ is monotone nondecreasing with respect to $r$, utilizing this property to examine the local regularity of harmonic functions and minimal surfaces. The monotonicity of $N_{e}(r)$ has also been crucial in the investigation of unique continuation properties of elliptic operators on Riemannian manifolds, as demonstrated by Garofalo and Lin \cite{Garofalo and Lin Fanghua,Garofalo and Lin Fanghua 1987} and in estimating the size of nodal sets of solutions to elliptic and parabolic equations, as shown by Lin \cite{Lin Fanghua}. For additional applications, refer to \cite{Colding and Minicozzi 1997,Han Qing and Lin Fanghua} and others.

In 1996, Poon \cite{Poon} introduced the parabolic frequency
$$N_{p}(t)=\dfrac{t\int_{\mathbb{R}^{n}} |\nabla u|^{2}(x, T-t)G(x,x_{0},t) dx}{\int_{\mathbb{R}^{n}} u^{2}(x, T-t)G(x,x_{0},t) dx},$$
where $u$ represents a solution to the heat equation on $\mathbb{R}^{n}\times[0,T]$ and $G(x, x_{0}, t)$ is the heat kernel with a pole at $(x_{0}, 0)$. He proved that $N_{p}(t)$ is monotone nondecreasing and derived several unique continuation results based on this property. Furthermore, Poon \cite{Poon} established the monotonicity of parabolic frequency on Riemannian manifolds under specific curvature conditions, a result that was independently verified by Ni \cite{Ni Lei}. Colding and Minicozzi \cite{Colding and Minicozzi 2022} further proved the monotonicity of parabolic frequency on manifolds using the drift Laplacian operator, without imposing any curvature or additional assumptions. Additionally, Li and Wang \cite{Li Xiaolong and Wang Kui} explored the parabolic frequency on compact Riemannian manifolds and in the context of the 2-dimensional Ricci flow by applying the matrix Harnack's inequality in \cite[Proposition 10.20]{Hamilton's Ricci Flow}.

For a general Ricci flow, Baldauf and Kim \cite{Baldauf and Kim-Ricci flow} defined the following parabolic frequency for a solution $u(t)$ of the heat equation:
$$N(t)=-\frac{(T-t)\parallel \nabla_{g(t)}u \parallel_{L^{2}(d\mu)}^{2}}{\parallel u \parallel_{L^{2}(d\mu)}^{2}}
\cdot \mathrm{exp}\left\{-\int_{t_{0}}^{t} \frac{1-\rho(s)}{T-s} ds\right\},$$
where $t \in [t_{0}, t_{1}] \subset (0, T )$, $\rho(t)$ represents a time-dependent function and $d\mu$ denotes the weighted measure. They proved that the parabolic frequency $N(t)$ is monotonically increasing along the Ricci flow when the Bakry-\'{E}mery Ricci curvature is bounded, thereby establishing the backward uniqueness. Additionally, Baldauf, Ho, and Lee derived analogous results under the mean curvature flow in \cite{Baldauf mean curvature flow}. We remind readers that further related research can be found in \cite{Li Xiaolong and Zhang Qi,Liu Hao Yue-frequency,Li Yi Ricci and Ricci harmonic flow}.

In \cite{Li Yi-G2 flow}, Li, Li, and Xu investigated the monotonicity of parabolic frequency under the Laplacian $G_{2}$ flow on manifolds. Specifically, they considered the monotonicity of parabolic frequency for solutions of the heat equation with bounded Ricci curvature.

Inspired by \cite{Li Yi-G2 flow}, we study the parabolic frequency for the solution of the heat equation (\ref{definition of heat eq.}) under the generalized Ricci flow (\ref{definition of generalized Ricii flow}) with bounded Ricci curvature. The parabolic frequency for the positive solution of the heat equation (\ref{definition of heat eq.}) is defined as follows:
$$U(t)=
\mathrm{exp}\{E(t)\}\frac{h(t)\int_{\mathrm{\bf{M}}} |\nabla_{g(t)}u|^{2}_{g(t)} d\mu_{g(t)}}{\int_{\mathrm{\bf{M}}} u^{2} d\mu_{g(t)}},$$
where the definition of $E(t)$ can be found in (\ref{frequency E(t)}), $h(t)$ is a time-dependent function. Utilizing
Theorem \ref{gradient estimates 1-1} and Theorem \ref{gradient estimates 1-2}, we obtain the following result.

\begin{theorem}\label{frequency 1-1}
Let $(\mathrm{\bf{M}}^{n},g(t),H(t)),t\in[0,T]$, be the solution of the generalized Ricci flow (\ref{definition of generalized Ricii flow}) on an $n$-dimensional closed manifold $\mathrm{\bf{M}}$ with $-\dfrac{K_1}{t}g \leq \mathrm{Ric} \leq \dfrac{K_2}{t}g$ and $H^2\leq\dfrac{K_{3}}{t}g$, where $K_{1}$, $K_{2}$, $K_{3}$ and $T<\infty$ are positive constants. Assume that $|\nabla H^{2}| \leq K_{4}$ for some constant $K_{4}>0$. Suppose that $u:\mathrm{\bf{M}}\times[0,T]\rightarrow\mathbb{R}$ is a smooth positive solution of the heat equation (\ref{definition of heat eq.}) with $\kappa \leq u(\cdot,0) \leq A$, then we have the following.
\begin{subequations}
\begin{align*}
     &\text{(i) if $h(t)<0$, then the parabolic frequency $U(t)$ is monotone increasing along the generalized}\\
     &\text{Ricci flow.}\\
     &\text{(ii) if $h(t)>0$, then the parabolic frequency $U(t)$ is monotone decreasing along the generalized}\\
     &\text{Ricci flow.}
\end{align*}
\end{subequations}
\end{theorem}

\begin{remark}
In particular, the forementioned results also apply to the Ricci flow, which shows that Theorem \ref{frequency 1-1} serves as an extension of Theorem 1.3 in \cite{Li Yi Ricci and Ricci harmonic flow}.
\end{remark}

The remainder of this article is organized as follows. In section \ref{section 3}, we prove the Li-Yau-type gradient estimate (Theorem \ref{gradient estimates 1-1}) and the Hamilton-type gradient estimate (Theorem \ref{gradient estimates 1-2}) under the generalized flow (\ref{definition of generalized Ricii flow}) with bounded Ricci curvature. As an application, we derive the Harnack inequality (Theorem \ref{Harnack}) on spacetime. In section \ref{section 4}, using Theorem \ref{gradient estimates 1-1} and Theorem \ref{gradient estimates 1-2}, we prove the the monotonicity of the parabolic frequency for the solution of the heat equation (\ref{definition of heat eq.}) under the generalized Ricci flow (\ref{definition of generalized Ricii flow}) with bounded Ricci curvature. Consequently, we obtain the integral-type Harnack inequality (Corollary \ref{integral Harnack}).

\section{Gradient estimates under generalized Ricci flow}\label{section 3}
In this section, we prove Li-Yau-type and Hamilton-type estimates for positive solutions to the heat equation (\ref{definition of heat eq.}) coupled with the generalized Ricci flow (\ref{definition of generalized Ricii flow}).Subsequently, we derive the Harnack inequality on spacetime. To facilitate the proof of Theorem \ref{gradient estimates 1-1}, we present the following Lemmas.

\begin{lemma}\label{gradient estimates lemma}
Let $(\mathrm{\bf{M}}^{n},g(t),H(t)),t\in[0,T]$, be the solution of the generalized Ricci flow (\ref{definition of generalized Ricii flow}) on an $n$-dimensional closed manifold $\mathrm{\bf{M}}$. Suppose that $u:\mathrm{\bf{M}}\times[0,T]\rightarrow\mathbb{R}$ is a smooth positive function satisfying the heat equation (\ref{definition of heat eq.}). Then for any given $\alpha\in\mathbb{R}$ and $f=\ln u$, the function
$$F=t\left(|\nabla f|^{2}-\alpha\partial_{t}f\right)$$
satisfies the equality
\begin{equation*}
\begin{aligned}
(\Delta-\partial_{t})F
=&-2\langle\nabla f,\nabla F\rangle+t\left(2|\nabla^{2}f|^{2}+2\alpha\langle\mathrm{Ric}-\dfrac{1}{4}H^{2},\nabla^{2}f\rangle\right)\\
&+t\left[2\alpha\mathrm{Ric}(\nabla f,\nabla f)-\dfrac{\alpha}{2}H^{2}(\nabla f,\nabla f)+\dfrac{1}{2}H^{2}(\nabla f,\nabla f)\right]\\
&+t\alpha\langle\nabla f,\dfrac{1}{4}\nabla|H|^{2}-\dfrac{1}{2}\mathrm{div}H^{2}\rangle
-\left(|\nabla f|^{2}-\alpha\partial_{t}f\right).
\end{aligned}
\end{equation*}
\end{lemma}

\begin{proof}[\bf{Proof}]

Since $f=\ln u$, we have
\begin{equation}\label{equation of f}
\begin{aligned}
\Delta f=\partial_t f - |\nabla f|^2.
\end{aligned}
\end{equation}
Using $$\Delta f=\frac{1}{\sqrt{G}} \partial_{i}\left(\sqrt{G} g^{ij} \partial_j f\right),$$
where $G=\mathrm{det}(g_{ij})$, and generalized Ricci flow equation, we conclude
\begin{equation}\label{evolution of Lf}
\begin{aligned}
\partial_t (\Delta f)
=2\langle\mathrm{Ric}-\dfrac{1}{4}H^2,\nabla^2 f \rangle+\dfrac{1}{4}\langle\nabla f, \nabla|H|^2\rangle
-\frac{1}{2}\langle\nabla f,\mathrm{div}H^2\rangle+\Delta(\partial_t f).
\end{aligned}
\end{equation}
According to Bochner formula, we note
\begin{equation}\label{bochner formula}
\Delta(|\nabla f|^2)=2|\nabla^2 f|^2 +2\mathrm{Ric}(\nabla f, \nabla f) + 2\langle \nabla f, \nabla \Delta f \rangle.
\end{equation}
Therefore, applying (\ref{equation of f}), (\ref{evolution of Lf}), and (\ref{bochner formula}) yields
\begin{equation}\label{laplace F}
\begin{aligned}
\Delta F
=&t\left( 2|\nabla^2 f|^2 + 2\mathrm{Ric}(\nabla f, \nabla f) + 2\langle \nabla f, \nabla \Delta f \rangle - \alpha \partial_t (\Delta f)
+ 2\alpha \langle \mathrm{Ric} - \frac{1}{4}H^2, \nabla^2 f \rangle \right.\\
&\left.+ \alpha \langle \nabla f, \frac{1}{4}\nabla |H|^2 - \frac{1}{2}\mathrm{div}H^2 \rangle \right).
\end{aligned}
\end{equation}
Again using (\ref{equation of f}) and generalized Ricci flow equation, we get
\begin{equation*}
\begin{aligned}
\partial_t (\Delta f)
=&\partial_t^2 f -\partial_t (|\nabla f|^2)\\
=&\partial_t^2 f -\left( 2 \mathrm{Ric}(\nabla f, \nabla f)-\frac{1}{2}H^2(\nabla f, \nabla f)+2\langle \nabla f, \nabla (\partial_t f) \rangle \right).
\end{aligned}
\end{equation*}
Substituting the above equation into (\ref{laplace F}), we now obtain
\begin{equation*}
\begin{aligned}
\Delta F =& t \left[ 2|\nabla^2 f|^2 + 2\mathrm{Ric}(\nabla f, \nabla f) + 2\langle \nabla f, \nabla \Delta f \rangle - \alpha \partial_t^2 f
+ 2\alpha\mathrm{Ric}(\nabla f, \nabla f) - \frac{\alpha}{2}H^2(\nabla f, \nabla f)\right.\\
&\left.+ 2\alpha \langle \nabla f, \nabla (\partial_t f) \rangle + 2\alpha \langle \mathrm{Ric} - \frac{1}{4}H^2, \nabla^2 f \rangle
+ \alpha \langle \nabla f, \frac{1}{4}\nabla |H|^2 - \frac{1}{2}\mathrm{div}H^2 \rangle \right].
\end{aligned}
\end{equation*}
Some computations show that
\begin{equation*}
\partial_t F
= |\nabla f|^2 - \alpha \partial_t f + t \left[ 2 \mathrm{Ric}(\nabla f, \nabla f) - \frac{1}{2}H^2(\nabla f, \nabla f)
+ 2\langle \nabla f, \nabla (\partial_t f) \rangle - \alpha \partial_t^2 f \right].
\end{equation*}
Combining the above two equations, we deduce that
\begin{equation*}
\begin{aligned}
(\Delta - \partial_t)F
=&t [ 2\langle \nabla f, \nabla \Delta f \rangle + 2\alpha \langle \nabla f, \nabla (\partial_t f) \rangle
- 2\langle \nabla f, \nabla (\partial_t f) \rangle ] \\
&+ t \left[ 2|\nabla^2 f|^2 + 2\alpha \langle \mathrm{Ric} - \frac{1}{4}H^2, \nabla^2 f \rangle \right]\\
&+ t \left[ 2\alpha \mathrm{Ric}(\nabla f, \nabla f) - \frac{\alpha}{2}H^2(\nabla f, \nabla f) + \frac{1}{2}H^2(\nabla f, \nabla f)\right]\\
&+ t\alpha \langle \nabla f, \frac{1}{4}\nabla |H|^2 - \frac{1}{2}\mathrm{div}H^2 \rangle - \left(|\nabla f|^2 - \alpha \partial_t f\right).
\end{aligned}
\end{equation*}
The Lemma is completed with the help of the equation
$$
t \left[ 2\langle \nabla f, \nabla \Delta f \rangle + 2\alpha \langle \nabla f, \nabla (\partial_t f) \rangle
- 2\langle \nabla f, \nabla (\partial_t f) \rangle \right]
=-2\langle \nabla f, \nabla F \rangle.
$$

\end{proof}

\begin{lemma}\label{grad(Q)}
Let $(\mathrm{\bf{M}}^{n}, g)$ be an n-dimensional Riemannian manifold. If $Q$ is a $2$-tensor, then the following holds
$$|\nabla(\mathrm{tr}Q)|^{2} \leq n|\nabla Q|^{2}.$$
\end{lemma}

\begin{proof}[\bf{Proof}]
$Q$ can be orthogonally decomposed as $$Q=\frac{\mathrm{tr}Q}{n}g+V,$$
where $V$ is totally trace-free.
Therefore we obtain
\begin{equation}\label{gradient of Q}
\begin{aligned}
\nabla Q=\frac{\nabla(\mathrm{tr}Q)}{n}\otimes g + \nabla V.
\end{aligned}
\end{equation}
By directly calculating using $\mathrm{tr}V=0$, it can be seen that
$$\left\langle \frac{\nabla(\mathrm{tr}Q)}{n}\otimes g, \nabla V \right\rangle =0.$$
Taking the squared norm of (\ref{gradient of Q}) and then using this property, we have
$$|\nabla Q|^{2}=\left|\frac{\nabla(\mathrm{tr}Q)}{n}\otimes g\right|^{2}+|\nabla V|^{2}.$$
According to $\left|\dfrac{\nabla(\mathrm{tr}Q)}{n}\otimes g\right|^{2}=\dfrac{1}{n}|\nabla(\mathrm{tr}Q)|^{2}$, we conclude that
\begin{equation*}
n|\nabla Q|^{2}=|\nabla(\mathrm{tr}Q)|^{2}+|\nabla V|^{2} \geq |\nabla(\mathrm{tr}Q)|^{2}.
\end{equation*}

\end{proof}

Using the Lemma \ref{gradient estimates lemma} and Lemma \ref{grad(Q)}, we prove Li-Yau-type estimate.

\begin{proof}[\bf{Proof of Theorem \ref{gradient estimates 1-1}}]

Let $f=\ln u$ and
$$F=t\left(|\nabla f|^{2}-\alpha\partial_{t}f\right).$$
Observe that, the Theorem \ref{gradient estimates 1-1} is true when $F \leq 0$, hence we always assume that $F > 0$ in the sequel.
By Lemma \ref{gradient estimates lemma}, we have
\begin{equation}\label{laplace-dt of F}
\begin{aligned}
(\Delta-\partial_{t})F
=&-2\langle\nabla f,\nabla F\rangle+t\left(2|\nabla^{2}f|^{2}+2\alpha\langle\mathrm{Ric}-\dfrac{1}{4}H^{2},\nabla^{2}f\rangle\right)\\
&+t\left[2\alpha\mathrm{Ric}(\nabla f,\nabla f)-\dfrac{\alpha}{2}H^{2}(\nabla f,\nabla f)+\dfrac{1}{2}H^{2}(\nabla f,\nabla f)\right]\\
&+t\alpha\langle\nabla f,\dfrac{1}{4}\nabla|H|^{2}-\dfrac{1}{2}\mathrm{div}H^{2}\rangle
-\left(|\nabla f|^{2}-\alpha\partial_{t}f\right).
\end{aligned}
\end{equation}
For the second term on the right-hand side of the equation (\ref{laplace-dt of F}), using the trick in \cite{Ricci Harmonic Flow} (\cite{cao xiaodong}), we have
\begin{equation*}
\begin{aligned}
&(a\alpha+2b\alpha)|\nabla^2f|^2+\alpha\langle \mathrm{Ric}, \nabla^2f \rangle-\frac{\alpha}{4}\langle H^2, \nabla^2f \rangle\\
=&a\alpha|\nabla^2f|^2+\alpha\left|\sqrt{b}\nabla^2f+\frac{1}{2\sqrt{b}}\mathrm{Ric}\right|^2
-\frac{\alpha}{4b}|\mathrm{Ric}|^2+\alpha\left|\sqrt{b}\nabla^2f-\frac{1}{8\sqrt{b}}H^2\right|^2-\frac{\alpha}{64b}\left|H^2\right|^2,
\end{aligned}
\end{equation*}
where $a, b>0$ are constants satisfying $a+2b=\dfrac{1}{\alpha}$.
We obtain
$$|\mathrm{Ric}|^2 \leq \frac{Kn}{t^{2}}$$
from $-\dfrac{K_1}{t}g \leq \mathrm{Ric} \leq \dfrac{K_2}{t}g$, where $K = \max\{K_1^2, K_2^2\}$.
Together with $H^2 \leq \dfrac{K_3}{t}g$, we get
$$\left|H^2\right|^2 \leq \frac{nK_3^2}{t^2}.$$
Noting that $|\nabla^{2}f|^{2} \geq \dfrac{1}{n}(\Delta f)^{2}$, thus we conclude that
\begin{equation}\label{the 2th term}
\begin{aligned}
t\left[2|\nabla^2f|^2+2\alpha\langle \mathrm{Ric}-\frac{1}{4}H^2, \nabla^2f \rangle\right]
\geq \frac{2a\alpha}{n}(\Delta f)^{2} t- \frac{\alpha Kn}{2bt} - \frac{\alpha K_3^2 n}{32bt}.
\end{aligned}
\end{equation}
Since $\mathrm{Ric} \geq -\dfrac{K_{1}}{t}g$ and $H^{2} \leq \dfrac{ K_{3}}{t}g$, if $\alpha >1$, the third term of equation (\ref{laplace-dt of F}) satisfies
\begin{equation}\label{the 3th term}
\begin{aligned}
t\left[2\alpha \mathrm{Ric}(\nabla f, \nabla f)-\frac{\alpha}{2}H^2(\nabla f, \nabla f) + \frac{1}{2}H^2(\nabla f, \nabla f)\right]
\geq -2\alpha K_1|\nabla f|^2-\frac{\alpha-1}{2}K_3|\nabla f|^2.
\end{aligned}
\end{equation}
For the fourth term, using Cauchy inequality, we deduce
\begin{equation*}
\langle \nabla f, \frac{1}{4}\nabla |H|^{2}-\frac{1}{2}\mathrm{div}H^2 \rangle
\geq -\frac{3}{8}|\nabla f|^{2}-\frac{1}{8}|\nabla|H|^2|^{2}-\frac{1}{4}|\mathrm{div}H^{2}|^{2}.
\end{equation*}
Moreover, Lemma \ref{grad(Q)} and direct computation assert that
$$|\nabla |H|^{2}|^2 \leq n|\nabla H^{2}|^{2},$$
$$|\mathrm{div}H^{2}|^2 \leq n|\nabla H^{2}|^2.$$
Together with $ |\nabla H^{2}| \leq K_4$, the fourth term on the right-hand side of equation (\ref{laplace-dt of F}) becomes
\begin{equation}\label{the 4th term}
\begin{aligned}
t\left[\alpha\langle \nabla f, \frac{1}{4}\nabla|H|^{2}-\frac{1}{2}\mathrm{div}H^2\rangle\right]
\geq -\frac{3\alpha}{8} t|\nabla f|^2 - \frac{3n\alpha K_{4}^{2}}{8} t.
\end{aligned}
\end{equation}
Substituting (\ref{the 2th term}), (\ref{the 3th term}), and (\ref{the 4th term}) into (\ref{laplace-dt of F}), we now obtain
\begin{equation}\label{laplace-dt of F again}
\begin{aligned}
(\Delta-\partial_{t})F
\geq &-2 \langle \nabla f, \nabla F \rangle + \frac{2a\alpha t}{n}\left(|\nabla f|^{2}-\partial_{t}f\right)^{2}
-\left(2\alpha K_{1}+\frac{(\alpha-1)K_{3}}{2}\right)|\nabla f|^{2}-\frac{3\alpha}{8}t|\nabla f|^{2}\\
&-\left(|\nabla f|^{2}-\alpha\partial_{t}f\right)
-\left(\frac{n\alpha K}{2b}+\frac{n\alpha K_{3}^{2}}{32b}\right)\frac{1}{t}-\frac{3n\alpha K_{4}^{2}}{8}t.
\end{aligned}
\end{equation}
Following the trick in \cite{cao xiaodong}(\cite{liu shiping}), the equality
$$(y-z)^2 = \frac{1}{\alpha^2}(y-\alpha z)^2+\left(\frac{\alpha-1}{\alpha}\right)^2y^2+\frac{2(\alpha-1)}{\alpha^2}y(y-\alpha z)$$
implies that
$$\left(|\nabla f|^{2}-\partial_{t}f\right)^{2}
=\frac{1}{\alpha^{2}}\left(|\nabla f|^{2}-\alpha\partial_{t}f\right)^{2}+\left(\frac{\alpha-1}{\alpha}\right)^{2}|\nabla f|^{4}
+\frac{2(\alpha-1)}{\alpha^{2}}|\nabla f|^{2}\left(|\nabla f|^{2}-\alpha\partial_{t}f\right).$$
Hence, we arrive at
\begin{equation}\label{transform quadratic}
\begin{aligned}
&\frac{2a\alpha t}{n} \left(|\nabla f|^2 - \partial_t f\right)^2 - \left(2\alpha K_1 + \frac{(\alpha-1)K_{3}}{2} \right) |\nabla f|^2
- \frac{3\alpha}{8} t |\nabla f|^2\\
=&\frac{2a\alpha t}{n}\left[\frac{1}{\alpha^{2}}\left(|\nabla f|^{2}
-\alpha\partial_{t}f\right)^{2}+\left(\frac{\alpha-1}{\alpha}\right)^{2}|\nabla f|^{4}
+\frac{2(\alpha-1)}{\alpha^{2}}|\nabla f|^{2}\left(|\nabla f|^{2}-\alpha\partial_{t}f\right)\right.\\
&\left.-\left(\frac{nK_1}{at}+\frac{n(\alpha-1)K_3}{4a\alpha t}+\frac{n}{16a}\right)|\nabla f|^{2}\right].
\end{aligned}
\end{equation}
According to the fundamental inequality
$$ c_{1}x^2-c_{2}x \geq -\frac{c_{2}^2}{4c_{1}}$$
for any $c_{1},c_{2} > 0$, we have
\begin{equation*}
\begin{aligned}
&\left(\frac{\alpha-1}{\alpha}\right)^{2}|\nabla f|^{4}
-\left(\frac{nK_1}{at}+\frac{n(\alpha-1)K_3}{4a\alpha t}+\frac{n}{16a}\right)|\nabla f|^{2}\\
\geq &-\frac{\alpha^2}{4(\alpha-1)^2}\left(\frac{nK_1}{at}+\frac{n(\alpha-1)K_3}{4a\alpha t}+\frac{n}{16a}\right)^2.
\end{aligned}
\end{equation*}
Applying the above inequality, (\ref{transform quadratic}) becomes
\begin{equation}\label{the key range of a}
\begin{aligned}
&\frac{2a\alpha t}{n} \left(|\nabla f|^2 - \partial_t f\right)^2 - \left(2\alpha K_1
+ \frac{(\alpha-1)K_{3}}{2} \right) |\nabla f|^2 - \frac{3\alpha}{8} t |\nabla f|^2\\
\geq &\frac{2a}{n\alpha} \frac{F^2}{t} - \frac{n\alpha^3}{2a(\alpha-1)^2} \left( K_1 + \frac{(\alpha-1)K_{3}}{4\alpha} \right)^2 \frac{1}{t}
- \frac{n\alpha^3}{512a(\alpha-1)^2} t \\
&- \frac{n\alpha^3}{16a(\alpha-1)^2} \left( K_1 + \frac{(\alpha-1)K_3}{4\alpha} \right),
\end{aligned}
\end{equation}
where we have used $$F=t\left(|\nabla f|^{2}-\alpha\partial_{t}f\right) \geq 0$$ and $$\alpha>1.$$
Therefore, combining the above inequality and (\ref{laplace-dt of F again}), we conclude that
\begin{equation*}
(\Delta-\partial_{t})F
\geq -2\langle \nabla f,\nabla F \rangle + \frac{2a}{n\alpha}\frac{F^{2}}{t}-\frac{F}{t}-B_{1}t-B_{2}\frac{1}{t}-B_{3},
\end{equation*}
where
$$B_{1}=\frac{n\alpha^3}{512a(\alpha-1)^2} + \frac{3n\alpha K_{4}^{2}}{8},$$
$$B_{2}=
\frac{n\alpha^3}{2a(\alpha-1)^2}\left(K_{1}+\frac{(\alpha-1)K_{3}}{4\alpha}\right)^{2} + \frac{n\alpha K}{2b} + \frac{n\alpha K_3^2}{32b},$$
$$B_{3}=\frac{n\alpha^3}{16a(\alpha-1)^2} \left( K_1 + \frac{(\alpha-1)K_3}{4\alpha} \right).$$
Fix $\sigma\in(0,T]$ and choose a point $(x_{0},t_{0})\in \mathbf{M}\times[0,\sigma]$ where $F$ attains its maximum on $\mathbf{M}\times[0,\sigma]$.
Noticing that $F=0$ as $t\rightarrow 0$, it means that the maximum value of \( F \) cannot be achieved at the initial moment by our positive assumption on $F$. Therefore, the following properties holds at the point $(x_{0},t_{0})$
$$\nabla F(x_{0},t_{0})=0,$$
$$\partial_{t}F(x_{0},t_{0}) \geq 0,$$
$$\Delta F(x_{0},t_{0}) \leq 0.$$
Evaluating at $(x_{0},t_{0})$ and using the above properties yields
$$\frac{2a}{n\alpha}\frac{F^{2}}{t}-\frac{F}{t}-B_{1}t-B_{2}\frac{1}{t}-B_{3} \leq 0.$$
Multiplying through by $t$ and the quadratic formula implies that
$$F \leq \frac{n\alpha}{4a}\left[1+\sqrt{1+\frac{8a}{n\alpha}(B_{1}t^{2}+B_{2}+B_{3}t)}~\right],$$
Since $F$ takes its maximum at $(x_{0},t_{0})$, for all $(x,t)\in \mathbf{M}\times [0,\sigma]$, we have the following estimate
$$F(x,t) \leq F(x_{0},t_{0}) \leq \frac{n\alpha}{2a}+\frac{\sqrt{2}}{2}\sqrt{\frac{n\alpha}{a}(B_{1}\sigma^{2}+B_{2}+B_{3}\sigma)}.$$
As $\sigma \in (0,T]$ was chosen arbitrarily, this holds for all $t \in (0,T]$.
According to the definition of $F$, we obtain the desired result
\begin{equation*}
\dfrac{|\nabla u|^{2}}{u^{2}}-\alpha\dfrac{\partial_{t}u}{u}
\leq\dfrac{n\alpha}{2at}+\sqrt{\dfrac{n\alpha B_{1}}{2a}}+\dfrac{1}{t}\sqrt{\dfrac{n\alpha B_{2}}{2a}}+\sqrt{\dfrac{n\alpha B_{3}}{2at}},
\end{equation*}
where $a+2b=\dfrac{1}{\alpha}$, $\alpha>1$.

\end{proof}

\begin{remark}
In particular, if $a=2b=\dfrac{1}{2\alpha}$, then the estimate becomes
\begin{equation}\label{a=2b}
\begin{aligned}
\dfrac{|\nabla u|^{2}}{u^{2}}-\alpha\dfrac{\partial_{t}u}{u}
\leq &\dfrac{n\alpha^{2}}{t}
+\sqrt{\frac{n^2 \alpha^6}{256(\alpha-1)^2} + \frac{3n^{2}\alpha^{3}K_{4}^{2}}{8}}\\
&+\dfrac{1}{t}\sqrt{\frac{n^2 \alpha^6}{(\alpha-1)^2}\left(K_{1}+\frac{(\alpha-1)K_{3}}{4\alpha}\right)^{2} + 2n^2 \alpha^4 K
+ \frac{n^2 \alpha^4 K_3^2}{8}}\\
&+\frac{1}{\sqrt{t}}\sqrt{\frac{n^2 \alpha^6}{8(\alpha-1)^2} \left( K_1 + \frac{(\alpha-1)K_3}{4\alpha} \right)},
\end{aligned}
\end{equation}
where $\alpha>1$. This shows that estimate (\ref{a=2b}) is an extension of Theorem 2 in \cite{liu shiping}.
\end{remark}

Next, we prove the Harnack-hype inequality.

\begin{proof}[\bf{Proof of Corollary \ref{Harnack}}]

From Theorem \ref{gradient estimates 1-1}, we have
\begin{equation}\label{gradient estimates corollary 3.1}
\begin{aligned}
\dfrac{|\nabla u|^{2}}{u^{2}}-\alpha\dfrac{\partial_{t}u}{u}
\leq \xi\frac{1}{t}+\frac{1}{\sqrt{t}}\sqrt{\frac{n\alpha}{2a}B_{3}}+\sqrt{\frac{n\alpha}{2a}B_{1}},
\end{aligned}
\end{equation}
where $$\xi=\frac{n\alpha}{2a}+\sqrt{\frac{n\alpha}{2a}B_{2}}.$$
Choosing a geodesic curve $\gamma(s):[0,1] \rightarrow \mathrm{\bf{M}}$ connects $x$ and $y$ with $\gamma(0)=y$ and $\gamma(1)=x$. Define
$$\eta(s)=\ln u(\gamma(s),(1-s)t_{2}+st_{1}),$$
then we obtain
$$\eta(0)=\ln u(y,t_{2}),\quad \eta(1)=\ln u(x,t_{1}).$$
Direct calculation gives
$$\partial_{s}\eta(s)=(t_{2}-t_{1})\left(\frac{1}{t_{2}-t_{1}}\langle \nabla\ln u, \gamma'(s) \rangle -(\ln u)_{t}\right),$$
where $t=(1-s)t_{2}+st_{1}$. The Cauchy inequality implies
$$\frac{1}{t_{2}-t_{1}}\langle \nabla\ln u, \gamma'(s) \rangle
\leq \frac{|\nabla u|^{2}}{\alpha u^{2}}+\frac{\alpha|\gamma'(s)|^{2}}{4(t_{2}-t_{1})^{2}}.$$
Combining the above inequality and (\ref{gradient estimates corollary 3.1}), we conclude that
$$\partial_{s}\eta(s) \leq \frac{\alpha|\gamma'(s)|^{2}}{4(t_{2}-t_{1})}
+\frac{t_{2}-t_{1}}{\alpha}\left[\xi\frac{1}{t}+\frac{1}{\sqrt{t}}\sqrt{\frac{n\alpha}{2a}B_{3}}+\sqrt{\frac{n\alpha}{2a}B_{1}}~\right],$$
where $t=(1-s)t_{2}+st_{1}$ and $|\gamma'(s)|$ is the length of the vector $\gamma'(s)$ at $t$.
Integrating this inequality over $\gamma'(s)$, we get
\begin{equation*}
\begin{aligned}
\ln\left(\frac{u(x,t_{1})}{u(y,t_{2})}\right)
&=\int_{0}^{1} \partial_{s}\eta(s) ds\\
& \leq \int_{0}^{1} \frac{\alpha|\gamma'(s)|^{2}}{4(t_{2}-t_{1})} ds
+\frac{\xi}{\alpha}\ln \left(\frac{t_{2}}{t_{1}}\right)+\frac{2}{\alpha}\sqrt{\frac{n\alpha}{2a}B_{3}}\left(\sqrt{t_{2}}-\sqrt{t_{1}}\right)
+\frac{t_{2}-t_{1}}{\alpha}\sqrt{\frac{n\alpha}{2a}B_{1}}.
\end{aligned}
\end{equation*}
The corollary follows by exponentiating both sides.
\end{proof}

At the end of this section, we provide the proof of the Hamilton-type inequality.

\begin{proof}[\bf{Proof} of Theorem \ref{gradient estimates 1-2}]
By calculating, we get
$$(\partial_{t}-\Delta)\left(u\ln\left(\frac{A}{u}\right)\right)=\frac{|\nabla u|^{2}}{u}.$$
Using the Bochner technique with the flow equation, we deduce
$$(\partial_{t}-\Delta)\left(\frac{|\nabla u|^{2}}{u}\right)
=-\frac{2}{u}\left|u_{ki}-\frac{u_{k}u_{i}}{u}\right|^{^{2}}-\frac{1}{2u}|i_{\nabla u}H|^{2}.$$
Denote$$P=t\frac{|\nabla u|^{2}}{u}-u\ln\left(\frac{A}{u}\right),$$ then
\begin{equation*}
\begin{aligned}
(\partial_{t}-\Delta)P
&=t\left(-\frac{2}{u}\left|u_{ki}-\frac{u_{k}u_{i}}{u}\right|^{^{2}}-\frac{1}{2u}|i_{\nabla u}H|^{2}\right)\\
& \leq 0.
\end{aligned}
\end{equation*}
Note that $P \leq 0$ as $t=0$, according to the maximum principle, we obtain $P \leq 0$ for all time. Then we prove this theorem.

\end{proof}

\section{Parabolic frequency on generalized Ricci flow}\label{section 4}
In this section, we first calculate the conjugate heat equation under the generalized flow (\ref{definition of generalized Ricii flow}).

For any time-dependent smooth function $f(t)$ on $\mathrm{\bf{M}}$, we denote $$\textbf{K}(t)=(4\pi(T-t))^{-\frac{n}{2}}e^{-f(t)}$$ as the positive solution of the conjugate heat equation
$$\partial_{t}\textbf{K}=-\Delta_{g(t)}\textbf{K}+R\textbf{K}-\frac{1}{4}|H|_{g(t)}^{2}\textbf{K}.$$
From the definition of $\textbf{K}(t)$, we can derive that the smooth function $f(t)$ satisfies the following equation
$$\partial_{t}f=-\Delta_{g(t)}f+|\nabla_{g(t)}f|_{g(t)}^{2}-R+\frac{1}{4}|H|_{g(t)}^{2}+\frac{n}{2(T-t)}.$$
We define the weighted measure as follows:
$$d\mu_{g(t)}:=\textbf{K}dV_{g(t)}=(4\pi(T-t))^{-\frac{n}{2}}e^{-f(t)}dV_{g(t)},~\int_{\mathrm{\bf{M}}}d\mu_{g(t)}=1.$$
Under the generalized Ricci flow (\ref{definition of generalized Ricii flow}), the volume form $dV_{g(t)}$ satisfies $$\partial_{t}(dV_{g(t)})=\left(-R+\frac{1}{4}|H|_{g(t)}^{2}\right)dV_{g(t)}.$$
Thus, the evolution of the conjugate heat kernel measure $d\mu_{g(t)}$ is given by
\begin{equation}\label{volume element}
\begin{aligned}
\partial_{t}(d\mu_{g(t)})=-\left(\Delta_{g(t)}\textbf{K}\right)dV_{g(t)}=-\frac{\Delta_{g(t)}\textbf{K}}{\textbf{K}}d\mu_{g(t)}.
\end{aligned}
\end{equation}

Then using the Li-Yau-type gradient estimate and the Hamilton-type gradient estimate, we study the parabolic frequency of the solution to the heat equation (\ref{definition of heat eq.}) under the generalized Ricci flow (\ref{definition of generalized Ricii flow}) with bounded Ricci curvature.

For a function $u: \mathrm{\bf{M}}\times[t_{0},t_{1}] \rightarrow \mathbb{R}_{+}$ and for all $t \in [t_{0},t_{1}] \subset (0,T)$, we define
$$I(t)=\int_{\mathrm{\bf{M}}} u^{2} d\mu_{g(t)},$$
$$D(t)=h(t)\int_{\mathrm{\bf{M}}} |\nabla_{g(t)} u|^{2}_{g(t)} d\mu_{g(t)},$$
$$U(t)=\mathrm{exp}\{E(t)\}\frac{D(t)}{I(t)},$$
where
\begin{equation}\label{frequency E(t)}
\begin{aligned}
E(t)=-\int_{t_{0}}^{t} \left(\frac{h'(s)}{h(s)}+\frac{4n}{s}+\frac{1}{s}\ln\left(\frac{A}{\kappa}\right)+\sqrt{4nC_{1}}
+\frac{\sqrt{4nC_{2}}}{s}+\frac{\sqrt{4nC_{3}}}{\sqrt{s}}+\frac{nc(s)}{2}\right) ds,
\end{aligned}
\end{equation}
$$C_{1}=\frac{1}{16} n + \frac{3}{4}nK_{4}^{2},$$
$$C_{2}=16n\left(K_1 + \frac{K_{3}}{8}\right)^{2} + 8nK + \frac{nK_3^2}{2},$$
$$C_{3}=2n \left( K_1 + \frac{K_3}{8} \right), \quad K=\mathrm{max}\{K_{1}^{2},K_{2}^{2}\}.$$
$$A=\mathrm{max}_{\mathrm{\bf{M}}}u(\cdot,0), \quad \kappa=\mathrm{min}_{\mathrm{\bf{M}}}u(\cdot,0), \quad c(t)=\frac{1}{t}\ln\left(\frac{A}{\kappa}\right),$$
the constants $K_{1},K_{2},K_{3},K_{4},K_{5}$ are as shown in Theorem \ref{gradient estimates 1-1}, and $h(t)$ is a time-dependent function.

In what follows, we will omit the subscript $g(t)$ in $I(t)$ and denote it simply as $I(t)=\int_{\mathrm{\bf{M}}} u^{2} d\mu$. Similar simplifications will be applied to other integrals where no confusion arises.

We now proceed to present the proof of the monotonicity of $U(t)$.

\begin{proof}[\bf{Proof of Theorem \ref{frequency 1-1}}]
Note that
$$U'(t) = E'(t) U(t) + \mathrm{exp}\{E(t)\} \frac{I(t) D'(t) - D(t) I'(t)}{I^2(t)}.$$
Hence, we need to calculate the derivative of $I(t)$ and $D(t)$.
From a direct calculation, it can be seen that
$$I'(t) = \frac{d}{dt}\left(\int_\mathrm{\bf{M}} u^2 \, d\mu\right)
= \int_\mathrm{\bf{M}} \left(2u\partial_{t}u-\Delta u^{2}\right) d\mu, $$
where we have used the divergence Theorem and (\ref{volume element}).
Using the equality $$\Delta u^{2}=2u\Delta u+2|\nabla u|^{2},$$
we obtain
\begin{equation}\label{equality I'(t)}
\begin{aligned}
I'(t)= \int_\mathrm{\bf{M}} 2\left(u\partial_{t}u-\frac{1}{2}|\nabla u|^{2}\right) d\mu - 2 \int_\mathrm{\bf{M}} u \Delta u d\mu
- \int_\mathrm{\bf{M}} |\nabla u|^2 d\mu.
\end{aligned}
\end{equation}
Taking $\alpha = 2$, $a = \dfrac{1}{4}$, $b = \dfrac{1}{8}$ in Theorem \ref{gradient estimates 1-1}, we get
\begin{equation}\label{frequency-estimate 1}
\begin{aligned}
\frac{|\nabla u|^{2}}{u^{2}}-2\frac{\partial_{t}u}{u} \leq \frac{4n}{t} + \sqrt{4nC_1}
+ \frac{\sqrt{4nC_2}}{t} + \frac{\sqrt{4nC_3}}{\sqrt{t}},
\end{aligned}
\end{equation}
where $$C_{1}=\frac{1}{16} n + \frac{3}{4}n K_4^2,$$
$$C_{2}=16n\left(K_1 + \frac{K_{3}}{8}\right)^{2} + 8nK + \frac{nK_3^2}{2},$$
$$C_{3}=2n \left( K_1 + \frac{K_3}{8} \right), \quad K=\mathrm{max}\{K_{1}^{2},K_{2}^{2}\}.$$
According to Theorem \ref{gradient estimates 1-2} and $\kappa \leq u(\cdot,0) \leq A$, we have
\begin{equation}\label{frequency-estimate 2}
\begin{aligned}
\int_\mathrm{\bf{M}} |\nabla u|^2 d\mu
\leq \frac{1}{t}\ln\left(\frac{A}{\kappa}\right)I(t).
\end{aligned}
\end{equation}
Combining (\ref{equality I'(t)}) with (\ref{frequency-estimate 1}) and (\ref{frequency-estimate 2}), we conclude that
\begin{equation}\label{inequality of I'(t)}
\begin{aligned}
I'(t) \geq -\mathcal{C}(t)\cdot I(t)- \frac{2}{nc(t)} \int_\mathrm{\bf{M}} (\Delta u)^{2} d\mu
\end{aligned}
\end{equation}
for $$\mathcal{C}(t)=\frac{4n}{t}+\frac{1}{t}\ln\left(\frac{A}{\kappa}\right)+ \sqrt{4nC_1} + \frac{\sqrt{4nC_2}}{t}+ \frac{\sqrt{4nC_3}}{\sqrt{t}} + \frac{n c(t)}{2},$$
where we have used the fundamental inequality
$$u\Delta u \leq \frac{nc(t)}{2}u^{2}+\frac{2}{nc(t)}(\Delta u)^{2},$$
and $c(t)$ is a time-dependent function to be determined later.

A straightforward calculation shows that
\begin{equation}\label{equality D'(t)}
\begin{aligned}
D'(t) =& \frac{d}{dt} \left[ h(t) \int_\mathrm{\bf{M}} |\nabla u|_{g(t)}^2 d\mu \right]\\
=& h'(t) \int_\mathrm{\bf{M}} |\nabla u|^2 d\mu + h(t) \int_\mathrm{\bf{M}} (\partial_{t} - \Delta) (|\nabla u|^2) d\mu.
\end{aligned}
\end{equation}
Using Bochner formula and generalized Ricci flow equation yields
$$(\partial_{t} - \Delta)(|\nabla u|^2)=-\frac{1}{2}|i_{\nabla u}H|^{2}-2|\nabla^{2}u|^{2}.$$
Substituting the above identity into (\ref{equality D'(t)}), we have
\begin{equation}\label{last equality D'(t)}
D'(t)= h'(t) \int_\mathrm{\bf{M}} |\nabla u|^2 d\mu -2 h(t) \int_\mathrm{\bf{M}} |\nabla^{2}u|^{2} d\mu
-\frac{1}{2}h(t) \int_\mathrm{\bf{M}} |i_{\nabla u}H|^{2} d\mu.
\end{equation}
If $h(t)<0$, then the equality (\ref{last equality D'(t)}) implies that
\begin{equation*}
D'(t) \geq h'(t) \int_\mathrm{\bf{M}} |\nabla u|^2 d\mu - 2 h(t) \int_\mathrm{\bf{M}} |\nabla^2 u|^2 d\mu.
\end{equation*}
Taking the above inequality together with (\ref{inequality of I'(t)}), (\ref{frequency-estimate 2}) and
$$|\nabla^2 u|^2 \geq \frac{1}{n}(\Delta u)^{2},$$
we obtain the estimate
\begin{equation*}
\begin{aligned}
I^2(t)U'(t)
\geq &~ \mathrm{exp}\{E(t)\} \left\{ I(t)[E'(t)h(t) + h'(t)] \int_\mathrm{\bf{M}} |\nabla u|^2 d\mu
+ I(t)h(t)\mathcal{C}(t)\int_\mathrm{\bf{M}} |\nabla u|^2 d\mu \right.\\
&\left.- \left[\frac{2}{n} I(t) h(t)
- \frac{2 h(t) I(t)}{t n c(t)} \ln\left(\frac{A}{\kappa}\right) \int_\mathrm{\bf{M}} (\Delta u)^2 d\mu \right]\right\}\\
=&~0
\end{aligned}
\end{equation*}
where we let $$c(t)=\frac{1}{t}\ln\left(\dfrac{A}{\kappa}\right).$$

On the other hand, if $h(t)>0$, then the equality (\ref{last equality D'(t)}) implies that
\begin{equation*}
D'(t) \leq h'(t) \int_\mathrm{\bf{M}} |\nabla u|^2 d\mu - 2 h(t) \int_\mathrm{\bf{M}} |\nabla^2 u|^2 d\mu.
\end{equation*}
Similarly, taking the this inequality together with (\ref{inequality of I'(t)}), we obtain
\begin{equation*}
\begin{aligned}
I^2(t)U'(t)
\leq ~& \mathrm{exp}\{E(t)\} \left[ E'(t) I(t) h(t) \int_\mathrm{\bf{M}} |\nabla u|^2 d\mu
+ I(t) h'(t) \int_\mathrm{\bf{M}} |\nabla u|^2 d\mu\right.\\
&\left.- 2 I(t) h(t) \int_\mathrm{\bf{M}} |\nabla^2 u|^2 d\mu
+ h(t)\mathcal{C}(t) I(t) \int_\mathrm{\bf{M}} |\nabla u|^2 d\mu\right.\\
&\left.+\frac{2h(t)}{n c(t)} \int_\mathrm{\bf{M}} (\Delta u)^2 d\mu \cdot \int_\mathrm{\bf{M}} |\nabla u|^2 d\mu\right].
\end{aligned}
\end{equation*}
Again applying (\ref{frequency-estimate 2}), the definition of $E(t)$ and $$|\nabla^2 u|^2 \geq \frac{1}{n}(\Delta u)^{2},$$
we find the estimate
\begin{equation*}
\begin{aligned}
I^2(t)U'(t)
\leq &~ \mathrm{exp}\{E(t)\} \left[-\frac{2}{n} I(t)h(t)\int_\mathrm{\bf{M}} (\Delta u)^2 d\mu
+ \frac{2 h(t) I(t)}{t n c(t)} \ln\left(\frac{A}{\kappa}\right) \int_\mathrm{\bf{M}} (\Delta u)^2 d\mu \right]\\
=&~0
\end{aligned}
\end{equation*}
where we let $$c(t)=\frac{1}{t}\ln\left(\dfrac{A}{\kappa}\right).$$
Therefore the result follows.

\end{proof}

We define the first non-zero eigenvalue of the manifold $\mathrm{\bf{M}}$ under the generalized Ricci flow with the weighted measure $d\mu_{g(t)}$ as
$$\lambda_{\mathrm{\bf{M}}}(t)
=\mathrm{inf}\left\{\frac{\int_{\mathrm{\bf{M}}} |\nabla_{g(t)}u|_{g(t)}^{2} d\mu_{g(t)}}{\int_{\mathrm{\bf{M}}} u^{2} d\mu_{g(t)}}\bigg|
0<u \in C^{\infty}(\mathrm{\bf{M}})\backslash\{0\}\right\}.$$
Consequently, we can derive the following Corollary from Theorem \ref{frequency 1-1}.
\begin{corollary}
Under the same hypotheses as Theorem \ref{frequency 1-1}, then for any $t \in [t_{0},t_{1}] \subset (0,T)$,we have the following.
\begin{subequations}
\begin{align*}
     &\text{(i) if $h(t)<0$, then $\beta(t)h(t)\lambda_{\mathrm{\bf{M}}}(t)$ is monotone increasing along the generalized Ricci flow.}\\
     &\text{(ii) if $h(t)>0$, then $\beta(t)h(t)\lambda_{\mathrm{\bf{M}}}(t)$ is monotone decreasing along the generalized Ricci flow.}
\end{align*}
\end{subequations}
where $\beta(t)=\mathrm{exp}\{E(t)\}$ and $E(t)$ is as shown in (\ref{frequency E(t)}).
\end{corollary}

\begin{corollary}\label{integral Harnack}
Under the same hypotheses as Theorem \ref{frequency 1-1}, then we have
$$I(t_{1}) \geq \mathrm{exp}\left\{2U(t_{0})\int_{t_{0}}^{t_{1}} -\frac{1}{h(t)\beta(t)} dt\right\}I(t_{0})$$
for any $t \in [t_{0},t_{1}] \subset (0,T)$.
\end{corollary}

\begin{proof}[\bf{Proof}]
According to the definition of $I(t)$ and $D(t)$, we can obtain
$$I'(t)=\frac{d}{dt}\left(\int_{\mathrm{\bf{M}}} u^{2 }d\mu\right)
=\int_{\mathrm{\bf{M}}} \left(2u\partial_{t}u-\Delta u^{2}\right)d\mu=-2\frac{D(t)}{h(t)}.$$
Therefore
$$\frac{d}{dt}[\ln(I(t))]=\frac{I'(t)}{I(t)}=-\frac{2U(t)}{h(t)\beta(t)}.$$
Integrating the above equality from $t_{0}$ to $t_{1}$, we have
\begin{equation}\label{4-4}
\ln(I(t_{1}))-\ln(I(t_{0}))=2\int_{t_{0}}^{t_{1}} -\frac{U(t)}{h(t)\beta(t)} dt.
\end{equation}
If $h(t)<0$, it follows from $\text{(i)}$ of Theorem \ref{frequency 1-1} that $U(t) \geq U(t_{0})$. Thus, using the equality (\ref{4-4}), we
get $$\ln(I(t_{1}))-\ln(I(t_{0})) \geq 2U(t_{0})\int_{t_{0}}^{t_{1}} -\frac{1}{h(t)\beta(t)} dt.$$
If $h(t)>0$, it follows from $\text{(ii)}$ of Theorem \ref{frequency 1-1} that $U(t) \leq U(t_{0})$. Thus, using the equality (\ref{4-4}), we
get $$\ln(I(t_{1}))-\ln(I(t_{0})) \geq 2U(t_{0})\int_{t_{0}}^{t_{1}} -\frac{1}{h(t)\beta(t)} dt.$$
The corollary follows by exponentiating both sides.
\end{proof}



\begin{thebibliography}{99}

\bibitem{Almgren}
Almgren, F. J.
``Dirichlet's problem for multiple valued functions and the regularity of mass minimizing integral currents."
In {\it Minimal Submanifolds and Geodesics},
1-6. Amsterdam-New York: North-Holland, 1979.

\bibitem{Ricci Harmonic Flow}
B\u{a}ile\c{s}teanu, M.
``Gradient estimates for the heat equation under the Ricci-harmonic map flow."
{\it Adv. Geom.}
15, no. 4 (2015): 445-454.

\bibitem{cao xiaodong}
B\u{a}ile\c{s}teanu, M., X. D. Cao, and A. Pulemotov.
``Gradient estimates for the heat equation under the Ricci flow."
{\it J. Funct. Anal.}
258, no. 10 (2010): 3517-3542.

\bibitem{Baldauf and Kim-Ricci flow}
Baldauf, J., and D. Kim.
``Parabolic frequency on Ricci flows."
{\it Int. Math. Res. Not. IMRN}
no. 12 (2023): 10098-10114.

\bibitem{Baldauf mean curvature flow}
Baldauf, J., P. T. Ho, and T.-K. Lee.
``Parabolic frequency for the mean curvature flow."
{\it Int. Math. Res. Not. IMRN}
no. 10 (2024): 8122-8136.

\bibitem{Hamilton's Ricci Flow}
Chow, B., P. Lu, and L. Ni.
{\it Hamilton's Ricci flow}.
Grad. Stud. Math., 77
American Mathematical Society, Providence, RI; Science Press Beijing, New York, 2006.

\bibitem{Colding and Minicozzi 1997}
Colding, T. H., and W. P. Minicozzi.
``Harmonic functions with polynomial growth."
{\it J. Differential Geom.}
46, no. 1 (1997): 1-77.

\bibitem{Colding and Minicozzi 2022}
Colding, T. H., and W. P. Minicozzi.
``Parabolic frequency on manifolds."
{\it Int. Math. Res. Not. IMRN}
no. 15 (2022): 11878-11890.

\bibitem{Garcia-generalized geometry}
Garcia-Fernandez, M.
``Ricci flow, Killing spinors, and T-duality in generalized geometry."
{\it Adv. Math.}
350, (2019): 1059-1108.

\bibitem{Streets-book}
Garcia-Fernandez, M., and J. Streets.
{\it Generalized Ricci flow}.
Univ. Lecture Ser., 76
American Mathematical Society, Providence, RI, 2021.

\bibitem{Garofalo and Lin Fanghua}
Garofalo, N., and F.-H. Lin.
``Monotonicity properties of variational integrals, $A_{p}$ weights and unique continuation."
{\it Indiana Univ. Math. J.}
35, no. 2 (1986): 245-268.

\bibitem{Garofalo and Lin Fanghua 1987}
Garofalo, N., and F.-H. Lin.
``Unique continuation for elliptic operators: a geometric-variational approach."
{\it Comm. Pure Appl. Math.}
40, no. 3 (1987): 347-366.

\bibitem{Hamilton 1993}
Hamilton, R. S.
``A matrix Harnack estimate for the heat equation."
{\it Comm. Anal. Geom.}
1, no. 1 (1993): 113-126.

\bibitem{Han Qing and Lin Fanghua}
Han, Q., R. Hardt, and F.-H. Lin.
`` Geometric measure of singular sets of elliptic equations."
{\it Comm. Pure Appl. Math.}
51, no. 11-12 (1998): 1425-1443.

\bibitem{entropy functional-Japan}
Ishida, M.
`` On the shrinking entropy functional for generalized Ricci flow."
{\it J. Geom. Anal.}
35, no. 5 (2025): 45, Paper No. 148.

\bibitem{Streets-Bochners formula}
Kopfer, E., and J. Streets.
`` Geometric measure of singular sets of elliptic equations."
{\it J. Funct. Anal.}
284,no. 10 (2023): 42, Paper No. 109901.

\bibitem{Li Yi Ricci and Ricci harmonic flow}
Li, C. H., Y. Li, and K. R. Xu.
``Parabolic frequency monotonicity on Ricci flow and Ricci-harmonic flow with bounded curvatures."
{\it J. Geom. Anal.}
33, no. 9 (2023): 21, Paper No. 282.

\bibitem{Li Yi-G2 flow}
Li, C. H., Y. Li, and K. R. Xu.
``Gradient estimates and parabolic frequency under the Laplacian $G_\text{2}$ flow."
{\it Calc. Var. Partial Differential Equations}
64, no. 4 (2025): 28, Paper No. 121.

\bibitem{Peter Li and Yau}
Li, P., and S.-T. Yau.
``On the parabolic kernel of the Schr\"{o}dinger operator."
{\it Acta Math.}
156, no. 3-4 (1986): 153-201.

\bibitem{Li Xiaolong and Wang Kui}
Li, X. L., and K. Wang.
`` Parabolic frequency monotonicity on compact manifolds."
{\it Calc. Var. Partial Differential Equations}
58, no. 6 (2019): 18, Paper No. 189.

\bibitem{Li Xiaolong and Zhang Qi}
Li, X. L., and Q. S. Zhang.
`` Matrix Li-Yau-Hamilton estimates under Ricci flow and parabolic frequency."
{\it Calc. Var. Partial Differential Equations}
63, no. 3 (2024): 38, Paper No. 63.

\bibitem{Li Xilun}
Li, X. L.
`` Entropy and heat kernel on generalized Ricci flow."
{\it J. Geom. Anal.}
34, no. 2 (2024): 23, Paper No. 42.

\bibitem{Lin Fanghua}
Lin, F.-H.
`` Nodal sets of solutions of elliptic and parabolic equations."
{\it Comm. Pure Appl. Math.}
44, no. 3 (1991): 287-308.

\bibitem{Liu Hao Yue-frequency}
Liu, H.-Y., and P. Xu.
``A note on parabolic frequency and a theorem of Hardy-P\'{o}lya-Szeg\"{o}."
{\it Internat. J. Math.}
33, no. 9 (2022): 15, Paper No. 2250064.

\bibitem{liu shiping}
Liu, S. P.
``Gradient estimates for solutions of the heat equation under Ricci flow."
{\it Pacific J. Math.}
243, no. 1 (2009): 165-180.

\bibitem{Moser 1964}
Moser, J.
``A Harnack inequality for parabolic differential equations."
{\it Comm. Pure Appl. Math.}
17, (1964): 101-134.

\bibitem{Ni Lei}
Ni, L.
{\it Parabolic frequency monotonicity and a theorem of Hardy-P\'{o}lya-Szeg\"{o}. In: Analysis, complexgeometry, and mathematical physics: in honor of Duong H. Phong}
Contemp. Math.,
American Mathematical Society, Providence, RI, 644, (2015): 203-210.

\bibitem{Oliynyk-math physics}
Oliynyk, T., V. Suneeta, and E. Woolgar.
``A gradient flow for worldsheet nonlinear sigma models."
{\it Nuclear Phys.}
B 739, no. 3 (2006): 441-458.

\bibitem{Poon}
Poon, C.-C.
``Unique continuation for parabolic equations."
{\it Comm. Partial Differential Equations}
21, no. 3-4 (1996): 521-539.

\bibitem{Streets-generalized geometry}
Streets, J.
``Generalized geometry, T-duality, and renormalization group flow."
{\it J. Geom. Phys.}
114, (2017): 506-522.

\bibitem{Streets-scalar and entropy}
Streets, J.
``Scalar curvature, entropy, and generalized Ricci flow."
{\it Int. Math. Res. Not. IMRN}
no. 11 (2023): 9481-9510.

\bibitem{Tian Gang-complex geometry}
Streets, J., and G. Tian.
`` A parabolic flow of pluriclosed metrics."
{\it Int. Math. Res. Not. IMRN}
no. 16 (2010): 3101-3133.

\bibitem{Tian Gang-generalized geometry}
Streets, J., and G. Tian.
`` Generalized K\"{a}hler geometry and the pluriclosed flow."
{\it Nuclear Phys.}
B 858, no. 2 (2012): 366-376.

\bibitem{Tian Gang and Streets-complex geometry}
Streets, J., and G. Tian.
`` Regularity results for pluriclosed flow."
{\it Geom. Topol.}
17, no. 4 (2013): 2389-2429.

\bibitem{Zhang-Hamilton estimate}
Zhang, Q. S.
``Some gradient estimates for the heat equation on domains and for an equation by Perelman."
{\it Int. Math. Res. Not.}
(2006): 39, Art. ID 92314.


\end{thebibliography}
\end{document}